\documentclass[11pt]{article}
\usepackage{amsthm, amsmath, amssymb, amsfonts, url, booktabs, tikz, setspace, fancyhdr, epstopdf}
\usepackage[margin = 1in]{geometry}
\usepackage{enumerate}
\usepackage{hyperref}

\usepackage[textsize=scriptsize,colorinlistoftodos]{todonotes}

\usepackage{enumitem}
\newlist{steps}{enumerate}{1}
\setlist[steps, 1]{label = Step \arabic*:}

\newtheorem{theorem}{Theorem}
\newtheorem{proposition}[theorem]{Proposition}
\newtheorem{lemma}[theorem]{Lemma}

\newtheorem{remark}[theorem]{Remark}

\theoremstyle{definition}

\newcommand{\G}{\mathcal{G}}

\newcommand{\HH}{\mathrm{H}}

\newcommand{\FF}{\mathcal{F}}
\newcommand{\TT}{\mathcal{T}}

\renewcommand{\S}{\mathcal{S}}

\newcommand{\Hom}{\mathrm{Hom}}

\newcommand{\ww}{\mathbf{w}}

\newcommand{\X}{\mathbf{X}}
\newcommand{\Y}{\mathbf{Y}}

\usetikzlibrary{trees}
\tikzstyle{p}+=[fill=black, circle, minimum width = 1pt, inner sep =
1pt]
\tikzstyle{w}+=[fill=white, draw, circle, minimum width = 1pt, inner sep =
1.5pt]

\begin{document}

\title{Sidorenko's conjecture for higher tree decompositions}

\author{
David Conlon\thanks{Mathematical Institute, Oxford OX2 6GG,
United Kingdom. Email: {\tt david.conlon@maths.ox.ac.uk}. Research
supported by a Royal Society University Research Fellowship.}
\and
Jeong Han Kim\thanks{School of Computational Sciences,  Korea Institute for Advanced Study (KIAS),  Seoul, South Korea. Email: {\tt jhkim@kias.re.kr.}~
The author was supported by the National Research Foundation of Korea (NRF) Grant funded by the Korean Government (MSIP) (NRF-2012R1A2A2A01018585) and KIAS internal Research Fund CG046001.}
\and
Choongbum Lee\thanks{Department of Mathematics,
MIT, Cambridge, MA 02139-4307. Email: {\tt cb\_lee@math.mit.edu.} Research supported  
by NSF Grant DMS-1362326.}
\and
Joonkyung Lee\thanks{
Mathematical Institute, University of Oxford, OX2 6GG, United Kingdom. 
Email: {\tt Joonkyung.Lee@maths.ox.ac.uk.} Supported by ILJU Foundation of Education and Culture.
}
}

\date{}

\maketitle

This is a companion note to \cite{CKLL15}, elaborating on a remark in that paper that the approach which proves Sidorenko's conjecture for strongly tree-decomposable graphs may be extended to a broader class, comparable to that given in Szegedy's work~\cite{Sz15}, through further iteration.

Very roughly, strongly tree-decomposable graphs are those that can be decomposed into trees in a certain tree-like way, a decomposition that facilitates the application of the entropy methods introduced to the study of Sidorenko's conjecture in \cite{LSz12} and developed further in \cite{KLL14}. In this note, we show that this idea can be iterated. For instance, graphs which can be decomposed into strongly tree-decomposable pieces in a tree-like way also satisfy Sidorenko's conjecture. Further iteration then results in a broad class of graphs satisfying the conjecture.

Before getting into the definition of these higher-order strong tree decompositions, we describe a probabilistic lemma that will
enable the use of the entropy calculus.
It will also help to explain why our definition is a natural one.

Let $\FF$ be a family of subsets in $[k]:=\{1,2,\cdots,k\}$.
Partly motivated by the notion of tree decomposition,
we define a \emph{Markov tree} on $[k]$ to be a pair $(\FF,\TT)$ consisting of a tree $\TT$ on vertex set $\FF$ that satisfies
\begin{enumerate}
	\item $\bigcup_{F\in\mathcal{F}}F=[k]$ and
	\item for $A,B,C\in \mathcal{F}$, $A\cap B\subseteq C$ 
whenever $C$ lies on the path from $A$ to $B$ in $\TT$.
\end{enumerate}
Let $V$ be a finite set and, for each $F \in \mathcal{F}$, let $(X_{i;F})_{i\in F}$ be a random vector indexed by pairs 
	$(i,F)$ with $i\in F$,
	taking values on $V^{F}$.
We are interested in collections of random vectors where `local' information 
is `globally' extendible.
More precisely, we want there to exist random variables $Y_1,Y_2,\cdots,Y_k$ such that, for each $F\in \FF$, the random vectors $(Y_i)_{i\in F}$ and $(X_{i;F})_{i\in F}$ are identically distributed over $V^{F}$.
If such $Y_1,\cdots,Y_k$ exist, then $(X_{i;A})_{i\in A\cap B}$ and $(X_{j;B})_{j\in A\cap B}$ must be identically distributed.
Our main probabilistic lemma states that the converse is also true and, moreover, the maximum entropy under such constraints can be attained. 

\begin{theorem}[\cite{L17}]\label{thm:tree_entropy}
	Let $(\FF,\TT)$ be a Markov tree on $[k]$ and 		let $(X_{i;F})_{i\in F}$ be random vectors taking values on a finite set $V^F$
	for each $F\in\FF$.  
	If $(X_{i;A})_{i\in A\cap B}$ and $(X_{j;B})_{j\in A\cap B}$ are identically distributed
	whenever $AB\in E(\TT)$,
	then there exist $Y_1,\cdots,Y_k$
	with entropy
	\begin{align}\label{eq:tree_entropy}
		\HH(Y_1,\cdots,Y_k)
		=\sum_{F\in\FF}\HH((X_{i;F})_{i\in F})
		-\sum_{AB\in E(\TT)}\HH((X_{i;A})_{i\in A\cap B})
	\end{align}
	such that $(Y_i)_{i\in F}$ and $(X_{i;F})_{i\in F}$ are identically distributed over $V^{F}$ for all $F\in\FF$.
\end{theorem}

The following version of the Kolmogorov Extension Theorem is the key ingredient in proving Theorem \ref{thm:tree_entropy}.

\begin{lemma}[Kolmogorov extension theorem]\label{lem:kolmogorov}
Let $(X_1,X_2)$ and $(X_2',X_3)$ be random vectors.
If $X_2$ and $X_2'$ are identically distributed, then there exists $(Y_1,Y_2,Y_3)$ 
such that
$Y_1$ and $Y_3$ are conditionally independent given $Y_2$
and, for $i=1,2,3$, $X_i$ and $Y_i$ are identically distributed.
\end{lemma}

The conclusion of this theorem is exactly what Szegedy \cite{Sz15} refers to as `conditional independent coupling'. Hence, what we describe in this note not only generalises the results in \cite{CKLL15},
but also explains what can be obtained by using the abstract recursive approach in \cite{Sz15}.
  
\begin{proof}[Proof of Theorem \ref{thm:tree_entropy}]
	We use induction on $|\FF|$.
	Fix a leaf $L$ of $\TT$ and let $\TT'$ be the tree $\TT\setminus L$ on $\FF':=\FF\setminus\{L\}$.
	By rearranging indices, we may assume that
	$L=\{t,t+1,\cdots,k\}$ for some $t\leq k$
	and that 
	$\FF'$ satisfies $\cup_{F\in\FF'}F=[\ell]$ for some $t\leq \ell\leq k$.
	By the inductive hypothesis, there is $\Y=(Y_1,Y_2,\cdots,Y_\ell)$ such that 
	$\Y_F:=(Y_i)_{i\in F}$ and $\X_F:=(X_{i;F})_{i\in F}$ are identically distributed
	for each $F\in\FF'$ and, moreover,
	\begin{align}\label{eq:induction}
		\HH(\Y)
		=\sum_{F\in\FF'}\HH((X_{i;F})_{i\in F})
		-\sum_{AB\in E(\TT')}\HH((X_{i;A})_{i\in A\cap B}).
	\end{align}
	Using Lemma \ref{lem:kolmogorov}
	with 
	\begin{align*}
	& X_1=(Y_1,Y_2,\cdots,Y_{t-1}),~~ X_2=(Y_t,Y_{t+1},\cdots,Y_\ell),\\
	& X_2'=(X_{t;L},X_{t+1;L},\cdots,X_{\ell;L}), X_3=(X_{\ell+1;L},X_{\ell+2;L},\cdots,X_{k;L}),
	\end{align*}
	we see that there exists
	$(Z_1,Z_2,\cdots,Z_k)$ such that
	$(Z_1,Z_2,\cdots,Z_{t-1})$ and $(Z_{\ell+1},\cdots,Z_{k})$ are conditionally independent given
	$(Z_t,Z_{t+1},\cdots,Z_{\ell})$,
	$Z_i$ and $Y_i$ are identically distributed for $i=1,2,\cdots,\ell$,
	and $Z_j$ and $X_{j;L}$ are identically distributed for all $t\leq j\leq k$.
	By conditional independence, we obtain
	\begin{align*}
		\HH(Z_1,Z_2,\cdots,Z_k)
		=\HH(\Y)+\HH(\X_{L})
		-\HH(Y_t,Y_{t+1},\cdots,Y_{\ell}).
	\end{align*}
	Using \eqref{eq:induction} and the fact that $\{t,t+1,\cdots,\ell\}=L\cap P$, 
	where $P$ is the neighbour of $L$ in $\TT$,
	\eqref{eq:tree_entropy} follows.
\end{proof}

\begin{remark}\label{rem:projection}
We say that $(\FF',\TT')$ is a \emph{Markov subtree} if $\TT'$ is a subtree of $\TT$ induced on the vertex set $\FF'\subseteq\FF$.
One may check that a Markov subtree of a Markov tree is again a Markov tree on the subset $K':=\cup_{F\in\FF'}F$ of $[k]$.
Thus, replacing $(\FF,\TT)$ by $(\FF',\TT')$ in Theorem \ref{thm:tree_entropy} gives
that there exists $(Y_i')_{i\in K'}$ satisfying
\begin{align*}
\HH((Y_i')_{i\in K'})=\sum_{F\in\FF'}\HH((X_{i;F})_{i\in F})
		-\sum_{AB\in E(\TT')}\HH((X_{i;A})_{i\in A\cap B})
\end{align*}
Moreover, the inductive proof above tells us 
that the random vectors $(Y_i)_{i\in K'}$ and $(Y_i')_{i\in K'}$ are identically distributed,
that is, the marginal distribution on $K'$ is preserved while extending $(Y_i')_{i\in K'}$ to $\Y=(Y_1,Y_2,\cdots,Y_k)$.
\end{remark} 

We are now ready for our recursive definition of $k$-\emph{strong tree decompositions}. For the base case, we define a $0$-strong tree decomposition of a tree $T$ to be $(\FF,\TT)$,
where $\FF=E(T)$ and $\TT$ is a spanning tree of the line graph of $T$.
For a positive integer $k$, a $k$-strong tree decomposition of a graph $H$ is a tree decomposition $(\FF,\TT)$ of $H$ together with a family $\S=(\FF_X,\TT_X,\S_X)_{X\in\FF}$ of $(k-1)$-strong tree decompositions that satisfy the following conditions:
\begin{enumerate}
\item Each $(\FF_X,\TT_X,\S_X)$ is a $(k-1)$-strong tree decomposition of the induced subgraph $H[X]$.
\item For each $XY\in E(\TT)$, $H[X\cap Y]$ is an induced forest.
\item For each $XY\in E(\TT)$, there is an isomorphism that fixes $X\cap Y$ between the minimum sub-decompositions containing $X\cap Y$ of the $(k-1)$-strong tree decompositions $(\FF_X,\TT_X,\S_X)$ and $(\FF_Y,\TT_Y,\S_Y)$. 
\end{enumerate}

To complete this recursive definition,
we must define sub-decompositions of $k$-strong tree decompositions and isomorphisms between them.
Firstly, a \emph{retraction} of a tree decomposition $(\FF,\TT)$ of a graph $H$ is a tree decomposition $(\FF',\TT')$ obtained by successively deleting leaves of $\TT$ to obtain $\TT'$ and letting $\FF'=V(\TT')$.
In particular, $(\FF',\TT')$ is a tree decomposition of the subgraph of $H$ induced on $\cup_{X\in \FF'}X$.
A \emph{sub-decomposition} of a $0$-strong tree decomposition $(\FF,\TT)$
is a retraction of $(\FF,\TT)$.
Thus, the underlying graph of a sub-decomposition of a $0$-strong tree decomposition $(\FF,\TT)$
is a subtree of the underlying tree of $(\FF,\TT)$.

To generalise this definition for positive $k$, we need the following proposition, whose proof we postpone to an appendix. 

\begin{proposition}\label{prop:subtree}
Let $(\FF,\TT)$ be a tree decomposition of $H$ and let $U\subseteq V(H)$. Given $u \in V(H)$, denote by $\FF(u)$ the subfamily $\{X\in\FF:u\in X\}$ of $\mathcal{F}$.
If $\cap_{u\in U}\FF(u)=\emptyset$, then
there exists a minimum subfamily $\FF'\subseteq\FF$ 
such that $\cup_{F\in\FF'}F$ contains $U$ and $\FF'$ induces a subtree of $\TT$.
\end{proposition}

Let $(\FF,\TT,\S)$ be a $k$-strong tree decomposition of a graph $H$
and let $U\subseteq V(H)$.
Suppose first that $\cap_{u\in U}\FF(u)$ is empty.
Then Proposition \ref{prop:subtree} implies that there exists a minimum $\FF'$
such that $\cup_{F\in\FF'}F$ contains $U$ and $\FF'$ induces a subtree in $\TT$.
The minimum sub-decomposition containing $U$ is then defined to be 
$(\FF',\TT',\S')$, where $(\FF',\TT')$ is a retraction of $(\FF,\TT)$.
Otherwise, if $\cap_{u\in U}\FF(u)$ is non-empty,
every vertex subset $X\in\cap_{u\in U}\FF(u)$ contains $U$.
We then define the minimum sub-decomposition of $(\FF,\TT,\S)$ containing $U$ to be
the minimum sub-decomposition of the $(k-1)$-strong tree decomposition 
$(\FF_X,\TT_X,\S_X)$ containing $U$ for some $X\in\cap_{u\in U}\FF(u)$.
In this case, $\FF'$ may not be unique, but the isomorphism condition
between $(k-1)$-strong tree decompositions implies uniqueness up to isomorphism.

To make this statement fully rigorous, we define an \emph{isomorphism} between two $k$-strong tree decompositions $(\FF_1,\TT_1,\S_1)$ of $H_1$ and
$(\FF_2,\TT_2,\S_2)$ of $H_2$ to be an isomorphism between the underlying graphs $H_1$ and $H_2$ that induces an isomorphism between each corresponding pair of $(k-1)$-strong tree decompositions in $\S_1$ and $\S_2$.
In particular, an isomorphism between two $0$-strong tree decompositions is simply a graph isomorphism between the underlying trees.
In summary, we see that sub-decompositions and isomorphisms of $k$-strong tree decompositions are both recursively defined in terms of the same notions for $(k-1)$-strong tree decompositions.

We say that a graph is $k$-strongly tree-decomposable if it is the underlying graph of a $k$-strong tree decomposition.
In particular, a graph is $0$-strongly tree-decomposable if and only if it is a tree.
Figure \ref{fig:decomp} gives an example of a $2$-strong tree decomposition. 
\begin{figure}
    \includegraphics[width=0.5\textwidth]{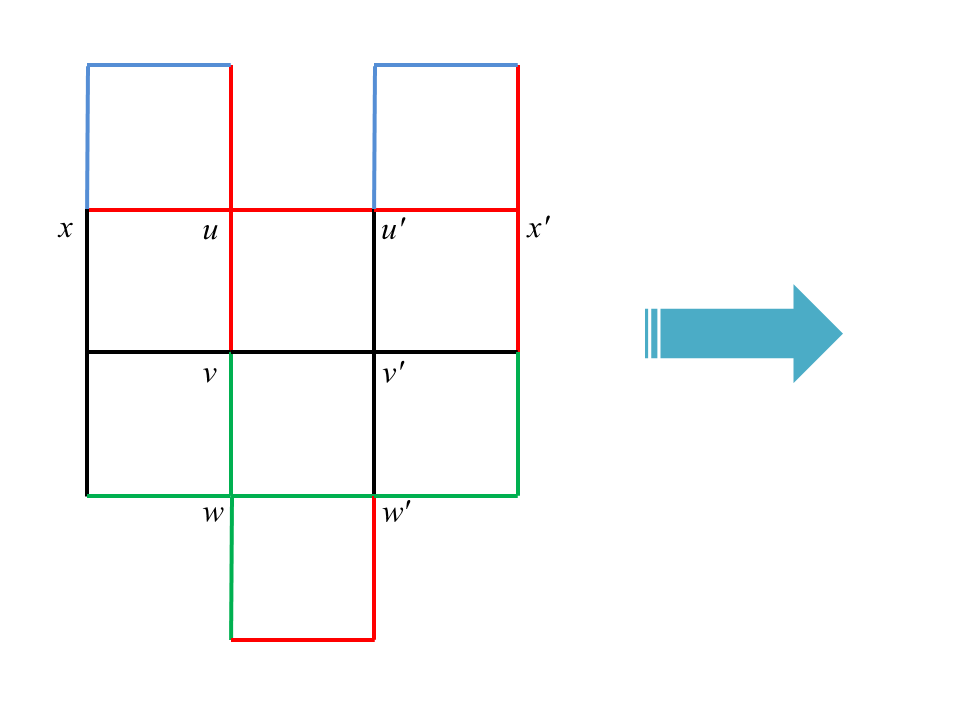}
    \includegraphics[width=0.5\textwidth]{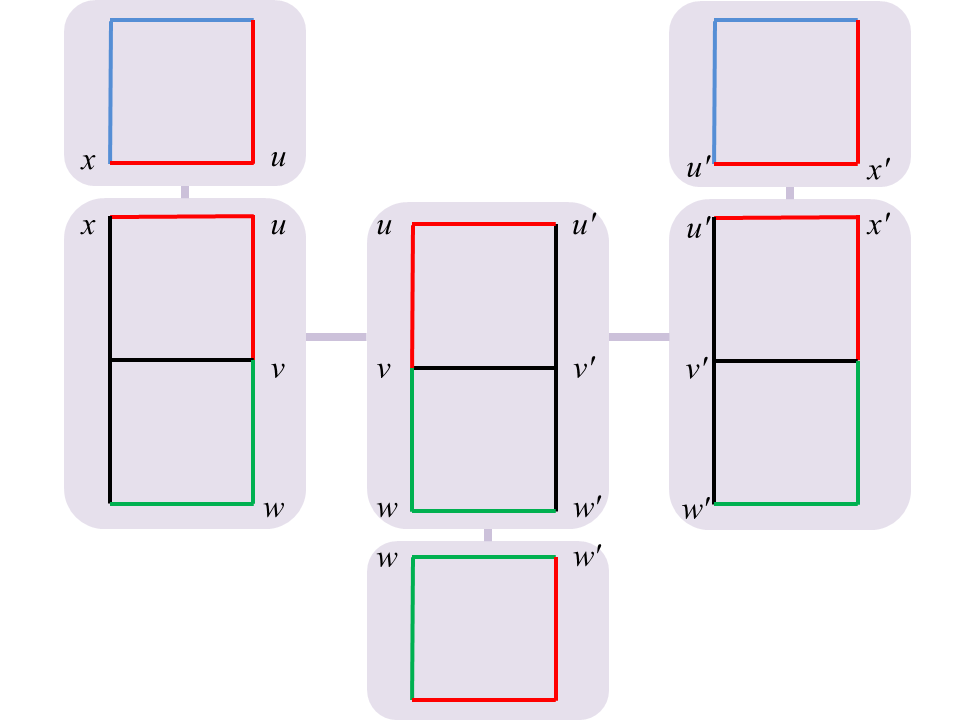}
    \caption{ An example of a $2$-strong tree decomposition.}\label{fig:decomp}
\end{figure}
The tree decomposition $(\FF,\TT)$ of $H$
is represented by the grey boxes,
in each of which the coloured trees show the 1-strong tree decomposition
of the corresponding induced subgraph.
For example, the underlying graph of a minimal sub-decomposition containing $u,v$, and $w$ is the subgrid with two squares induced by the leftmost grey box containing them or the isomorphic one next to it.
On the other hand, in the same box, the minimum sub-decomposition containing $x$,$u$, and $v$ is the trivial tree decomposition of the red tree induced on the three vertices with $\FF'=\{\{x,u\},\{u,v\}\}$ and the single edge tree $\TT'$.
 
The main result of this note is:

\begin{theorem}\label{thm:main_general}
If $H$ is $k$-strongly tree decomposable for any $k\geq 0$, then $H$ has Sidorenko's property. 
\end{theorem}

This generalises Theorem 1.2 in \cite{CKLL15},
as every strongly tree decomposable graph is $1$-strongly tree decomposable.
In fact, a $1$-strong tree decomposition is a slightly more general 
concept than a strong tree decomposition:
in the definition of strong tree decomposition in \cite{CKLL15},
$H[X\cap Y]$ was an independent set whenever $X$ and $Y$ were adjacent vertices of $\FF$, whereas here we allow $H[X\cap Y]$ to be a forest.
This technical improvement was first observed by Szegedy \cite[Lemma 3.1]{Sz15}.
We show that it can be recovered by using the following elementary lemma due to Fox (see~\cite{KLL14}).

\begin{lemma} \label{lem:sido_maxdegree}
A bipartite graph $H$ has Sidorenko's property
if and only if there exists a constant $c_H>0$ depending only on $H$ such that,
for all graphs $G$ with maximum degree at most $\frac{4|E(G)|}{|V(G)|}$, 
\[|\Hom(H,G)|\ge c_H |V(G)|^{|V(H)|}\left(\frac{2|E(G)|}{|V(G)|^{2}}\right)^{|E(H)|}.\]
\end{lemma}
Thus, in order to prove Sidorenko's property, it suffices to consider graphs $G$ whose maximum degree is close to the average degree.

\medskip

Let $(\FF,\TT)$ be a Markov tree and let $(X_{i;F})_{i\in F}$ be random vectors as in Theorem~\ref{thm:tree_entropy}.
We say that the distribution of $(Y_1,\cdots,Y_k)$ given by Theorem~\ref{thm:tree_entropy} is \emph{associated} with $(\FF,\TT)$ and $(X_{i;F})_{i\in F}$.
Let $(\FF,\TT,\S)$ be a $k$-strong tree decomposition $(\FF,\TT,\S)$ of $H$.
Then the distribution on $\Hom(H,G)$ that is \emph{$k$-associated} with $(\FF,\TT,\S)$ is defined recursively as follows:
the $0$-associated distribution is the one constructed by the $H$-BRW in~\cite{CKLL15};
if $k>0$, then the $k$-associated distribution 
is the one associated with the Markov tree $(\FF,\TT)$ and $(\ww_{i;X})_{i\in X}$, where $(\ww_{i;X})_{i\in X}$ is a random homomorphism in $\Hom(H[X],G)$ 
whose distribution is $(k-1)$-associated with $(\FF_X,\TT_X,\S_X)$.
That is, we extend the `local' distributions on $\Hom(H[X],G)$ given by lower complexity decompositions to a `global' one on $\Hom(H,G)$ using Theorem~\ref{thm:tree_entropy}.

To complete this recursive definition, we must prove that the distributions on $\Hom(H[X],G)$ and $\Hom(H[Y],G)$, $XY\in E(\TT)$, that are $(k-1)$-associated
with $(\FF_X,\TT_X,\S_X)$ and $(\FF_Y,\TT_Y,\S_Y)$, respectively, 
agree on the marginals on the intersection $X\cap Y$. Knowing this will enable us to apply Theorem~\ref{thm:tree_entropy}.
To facilitate our argument, we need some further definitions.
Let $p:\Hom(H,G)\rightarrow [0,1]$ be a weight function on $\Hom(H,G)$. For a subgraph $H'$ of $H$,
the \emph{$H'$-projection} of $p$ is the natural function $q:\Hom(H',G)\rightarrow[0,1]$ defined by $q(h')=\sum_{h\in \G(h')}p(h)$,
where $\G(h')$ is the set of all homomorphisms from $H$ to $G$ that extend $h'$.
This projection is transitive, i.e., for a subgraph $H''$ of $H'$,
one may easily check that if $q$ is the $H'$-projection of $p$ and $r$ is the $H''$-projection of $q$,
then $r$ is the $H''$-projection of $p$.
Let $\phi:V(H_2)\rightarrow V(H_1)$ be an isomorphism between the two graphs $H_1$ and $H_2$.
Then a weight function $p_1:\Hom(H_1,G)\rightarrow[0,1]$ naturally defines another weight function $p_2=p_1\circ\phi:\Hom(H_2,G)\rightarrow[0,1]$.
If two distributions $p_1$ and $p_2$ satisfy $p_2=p_1\circ\phi$, then
we say that the two are \emph{preserved} by the isomorphism $\phi$.

\begin{lemma}\label{lem:commutes}
For positive integers $\ell$ and $k$ with $\ell\leq k$,
let $(\FF,\TT,\S)$ and $(\FF',\TT',\S')$ be a $k$-strong tree decomposition of $H$ and an $\ell$-strong tree decomposition of $H'$, respectively. Then the following hold:

\vspace{1mm}
\noindent
(i) if $H'$ is a subgraph of $H$ and $(\FF',\TT',\S')$ is a sub-decomposition of $(\FF,\TT,\S)$, then the $H'$-projection of the $k$-associated distribution on $\Hom(H,G)$ is identical to the $\ell$-associated distribution on $\Hom(H',G)$;

\vspace{1mm}
\noindent
(ii) if $(\FF',\TT',\S')$ and $(\FF,\TT,\S)$ are isomorphic, then the $k$-associated distributions of the two are preserved by the isomorphism.
\end{lemma}
\begin{proof}
We use induction on $k$.
Proposition 2.2 and Equation (2) in \cite{CKLL15} imply (i) and (ii), respectively, for the base case $k=0$.
The induction hypothesis, that (i) and (ii) hold for any $(k-1)$-strong tree decomposition $(\FF,\TT,\S)$ of a graph $H$ and any $\ell$-strong tree decomposition $(\FF',\TT',\S')$ of a graph $H'$
with $\ell\leq k-1$,
implies that the $k$-associated distribution of a $k$-strong tree decomposition is well-defined.
Suppose now that $H'$ is a subgraph of $H$. 
We claim that the $H'$-projection of the $k$-associated distribution $p$ on $\Hom(H, G)$ is identical to the $\ell$-associated distribution $q$ on $\Hom(H',G)$.
Suppose $\ell<k$, i.e., there exists $X\in\FF$ such that $(\FF',\TT',\S')$ is a sub-decomposition of $(\FF_X,\TT_X,\S_X)$.
In particular, $H'$ is a subgraph of $H[X]$.
Then, by the induction hypothesis, $q$ is identical
to the $H'$-projection of the $(k-1)$-associated distribution $p_X$ of $(\FF_X,\TT_X,\S_X)$.
Since $p_X$ is the $H[X]$-projection of $p$, 
transitivity implies that $q$ is the $H'$-projection of $p$.
Suppose now that $\ell=k$. Then $(\FF',\TT')$ is a retraction of $(\FF,\TT)$.
Hence, by Remark~\ref{rem:projection}, $q$ is the $H'$-projection of $p$.
It remains to prove (ii). Let $\phi$ be an isomorphism from $(\FF,\TT,\S)$ to $(\FF',\TT',\S')$.
By induction, $\phi$ gives an isomorphism between
the $(k-1)$-strong tree decompositions $(\FF_X,\TT_X,\S_X)$ and $(\FF_Y,\TT_Y,\S_Y)$ for every $X\in \FF$ and $Y=\phi(X)\in\FF'$.
Thus, the two $k$-associated distributions of $(\FF,\TT,\S)$ and $(\FF',\TT',\S')$ are
associated with the same $(k-1)$-associated distributions and have isomorphic Markov trees.
Therefore, they are constructed in the same way by using Theorem~\ref{thm:tree_entropy} and hence preserved by $\phi$.
\end{proof}

Suppose now that we have a $k$-strong tree decomposition $(\FF,\TT,\S)$ of a graph $H$ and $XY\in E(\TT)$. By definition, we know that there is an isomorphism that fixes $X\cap Y$ between the minimum sub-decompositions containing $X\cap Y$ of the $(k-1)$-strong tree decompositions $(\FF_X,\TT_X,\S_X)$ and $(\FF_Y,\TT_Y,\S_Y)$. By Lemma~\ref{lem:commutes}(i), the projected distribution on each of these minimum sub-decompositions is the associated distribution. Moreover, by Lemma~~\ref{lem:commutes}(ii), since the two sub-decompositions are isomorphic, the distributions are preserved by this isomorphism and so agree on $X \cap Y$. Hence, this completes the definition of $k$-associated distribution.

\begin{proof}[Proof of Theorem \ref{thm:main_general}]
By Lemma~\ref{lem:sido_maxdegree}, our goal is to prove that, for each $k$-strongly tree decomposable graph $H$ and every graph $G$ with maximum degree at most $\frac{4|E(G)|}{|V(G)|}$,
there exists a random homomorphism $Y=(Y_1,Y_2,\cdots,Y_{|V(H)|})$ on $\Hom(H,G)$ such that
\begin{align}\label{eq:rightbound}
\HH(Y) \geq |E(H)|\log \left(\frac{2|E(G)|}{|V(G)|^2}\right)+|V(H)|\log |V(G)| +c_H,
\end{align}
where $c_H$ depends only on $H$.
If this holds, then the fact that $\HH(Y)\leq\log|\Hom(H,G)|$ verifies the hypothesis of Lemma~\ref{lem:sido_maxdegree} and, hence, Sidorenko's conjecture holds for $H$.

We claim that for each $k$, every random homomorphism $Y$ on $\Hom(H,G)$
with the $k$-associated distribution of the $k$-strong tree decomposition $(\FF,\TT,\S)$ of $H$
satisfies \eqref{eq:rightbound}.
The proof is by induction on $k$. Suppose that the claim holds up to $k-1$.
Denote by $e_A$ the number of edges induced on $A\subseteq V(H)$. 
Theorem~\ref{thm:tree_entropy} (which may be applied because of the argument above)  and the induction hypothesis then imply that, for some constant $c'_H$, 
\begin{align}\nonumber
		\HH(Y_1,\cdots,Y_k)
		&=\sum_{F\in\FF}\HH((X_{i;F})_{i\in F})
		-\sum_{AB\in E(\TT)}\HH((X_{i;A})_{i\in A\cap B})\\ 
		\label{eq:calculus}
		&\geq 
		\log \left(\frac{2|E(G)|}{|V(G)|^2}\right)
		\sum_{F\in\FF}
		e_F+\log |V(G)|\sum_{F\in\FF}|F| + c'_H
		-\sum_{AB\in E(\TT)}\HH((X_{i;A})_{i\in A\cap B}).
	\end{align} 
Since $H[A\cap B]$ is always a forest,
the fact that $G$ has maximum degree at most $\frac{4|E(G)|}{|V(G)|}$  
implies that
\begin{align*}
|\Hom(H[A\cap B],G)|&\leq 
|V(G)|^{|A\cap B|}\left(\frac{4|E(G)|}{|V(G)|^2}\right)^{e_{A\cap B}}\\
&=2^{e_{A\cap B}}|V(G)|^{|A\cap B|}\left(\frac{2|E(G)|}{|V(G)|^2}\right)^{e_{A\cap B}}.
\end{align*}
Taking logs, we obtain the upper bound
\begin{align*}
	\HH((X_{i;A})_{i\in A\cap B})
	\leq |A\cap B|\log |V(G)| +e_{A\cap B}\log \left(\frac{2|E(G)|}{|V(G)|^2}\right)+e_{A\cap B}.
\end{align*}
Plugging this into \eqref{eq:calculus} yields the desired bound~\eqref{eq:rightbound}.
\end{proof}

\bibliographystyle{abbrv}
\bibliography{references}

\section*{Appendix}

The following simple fact is immediate from the definition of tree decompositions.

\begin{lemma}\label{lem:subtree}
Let $(\FF,\TT)$ be a tree decomposition of $H$ and let $v\in V(H)$.
Denote by $\FF(v)$ the subfamily $\{X\in \FF: v\in X\}$ of $\FF$.
Then $\FF(v)$ always induces a subtree of $\TT$.
\end{lemma}

We also need the following folklore lemma.

\begin{lemma}\label{lem:helly}
Let $T$ be a tree and let $T_1, \dots, T_k$ be subtrees which pairwise intersect. Then $\cap_{i=1}^k T_i$ is non-empty.
\end{lemma}

Using these, we give a proof of Proposition~\ref{prop:subtree}.

\begin{proof}[Proof of Proposition~\ref{prop:subtree}]
By Lemma~\ref{lem:subtree}, $\FF(u)=\{X\in\FF:u\in X\}$ induces a subtree of $\TT$.
Let $U=\{u_1,u_2,\cdots,u_t\}$. We use induction on $t$.
For brevity, we say that $\FF'$ is \emph{good} for $U$ if $\cup_{F\in\FF'}F$ contains $U$ and $\FF'$ induces a subtree in $\TT$.
Suppose $t=2$ and $\FF(u_1)$ and $\FF(u_2)$ are disjoint.
Then the vertex set of the shortest path between the two vertex-disjoint 
trees $\TT[\FF(u_1)]$ and $\TT[\FF(u_2)]$ is the unique minimal good subfamily $\FF'$ for $U$.

For $t>2$, suppose that $\cap_{i=1}^{t-1}\FF(u_i)=\emptyset$.
Then, by the induction hypothesis, there exists a minimum good subfamily $\FF''\subseteq\FF$ for $U\setminus\{u_t\}$.
If $\FF''$ intersects $\FF(u_t)$ then let $\FF':=\FF''$.
This $\FF'$ is good for $U$ 
and, moreover, it is the minimum such family,
since every good subfamily for $U$ is again good for $U\setminus\{u_t\}$.  
Otherwise, if $\FF''$ and $\FF(u_t)$ are disjoint,
add the vertex set of the shortest path between $\FF''$ and $\FF(u_t)$ to $\FF''$.
Then the new subfamily $\FF'$ induces a tree and hence is good for $U$.
It is also minimal and unique, because every good subfamily $\FF_0$ for $U$ must contain $\FF''$ and, therefore, the shortest path from $\FF''$ to $\FF(u_t)$.

Suppose now that $\cap_{i=1}^{t-1}\FF(u_i)$ is non-empty. 
Let $\FF'$ be the vertex set of the shortest path from $\cap_{i=1}^{t-1}\FF(u_i)$
to $\FF(u_t)$.
We claim that $\FF'$ is the desired minimum good subfamily $\FF'$ for $U$.
As each good subfamily $\FF''$ for $U$ intersects every $\FF(u_i)$, $1\leq i\leq t$, Lemma~\ref{lem:helly} implies that $\FF''$ also intersects both $\cap_{i=1}^{t-1}\FF(u_i)$ and $\FF(u_t)$.
Thus, every $\FF''$ good for $U$ must contain $\FF'$.
It remains to prove that $U$ is contained in $\cup_{F\in\FF'}F$, as 
$\FF'$ already induces a path in $\TT$.
Firstly, the end vertex of the path $\FF'$ is in $\FF(u_t)$ and hence contains $u_t$. On the other hand, since every element of $\cap_{i=1}^{t-1}\FF(u_i)$ contains $U\setminus\{u_t\}$,
the starting vertex of $\FF'$ contains $U\setminus\{u_t\}$.
\end{proof}

\end{document}